\documentclass{article}
\usepackage{amsmath, amssymb, amsthm,graphicx}

\newtheorem{definition}{Definition}

\newtheorem{theorem}{Theorem}

\newtheorem{proposition}{Proposition}
\newtheorem{lemma}{Lemma}

\title{Variational proof of the existence of  the super-eight orbit in the four-body problem
}
\author{Mitsuru Shibayama\thanks{Mathematical Science, Osaka University; 
1-3 Machikaneyama-cho, Toyonaka, Osaka 560-8531, Japan (e-mail: shibayama@sigmath.es.osaka-u.ac.jp)}}
\begin{document}
\maketitle

\begin{abstract}

Using the variational method, 
Chenciner and Montgomery (2000 {\it Ann. Math.} {\bf 152} 881--901)  
proved the existence of an eight-shaped periodic solution of the planar  three-body problem 
with equal masses. 
Just after the discovery, 
Gerver have numerically found a similar periodic solution called ``super-eight''  in the planar four-body problem with equal 
mass. 
 
In this paper we prove the existence of the  super-eight orbit by using the variational method.
The difficulty of the proof is to eliminate the possibility of  collisions. 
In order to solve it, we apply the scaling technique established by Tanaka
(1993 {\it Ann. Inst. H. Poincar\'{e} Anal. Non Lin\'{e}aire} {\bf 10}, 215--238, 
1994 {\it Proc. Amer. Math. Soc.} {\bf 122}, 275--284) and investigate 
the asymptotic behavior of a binary collision.

\end{abstract}

\section{Main result}

Using the variational method, Chenciner and Montgomery \cite{CM00}
proved the existence of a new periodic solution of figure-eight shape to the planar three-body problem
with equal masses.
The solution is called figure-eight solution.
In that paper there is Simo's  numerical  solution  which is expected to be identical
with the theoretically obtained solution.

Super-eight orbit is a periodic orbit in the  planar four-body problem with equal masses,
along which 
 the particles shadow each other on the symmetric curve
as per Figure \ref{fig:s8}.
\begin{figure}[htbp]\label{fig:s8}
\begin{center}
\includegraphics[scale=0.7]{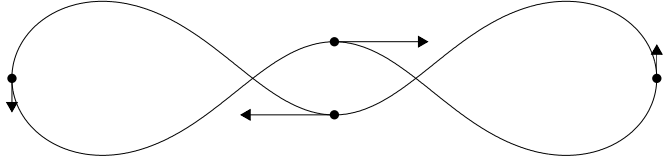}
\end{center}
\caption{The super-eight orbit}
\end{figure} 
Gerver  numerically discovered  this orbit just after the Chenciner-Mongomery's result.
A few years after, Kapela-Zgliczy\'nski \cite{KZ03} gave a computer-assisted proof of the existence.
Since Chenciner-Montgomary's result, 
a number of periodic solutions have been proved to exist.
But a variational proof of the existence of the super-eight has not devoted yet.
The purpose of this paper is to provide a variational proof of the existence of 
a periodic solution which can be thought to be identical with 
Gerver's numerical solution.
We also use the minimizing mehod of the Lagrangian action functional.
The difficulty in this field is always to eliminate a collision.
In order to solve that, 
the following three method have been used.
\begin{description}
\item[Global estimate] 
One estimates the lower bound of the action functional for collision paths
and the upper bound of the action functional for a test path, 
and then shows that the former is greater than the latter.
The test path is quite different from collision paths.
The technique  was used for getting several solutions
(for instance, \cite{CM00, Chen01}). 

\item[Local estimate] 
The Sundman estimate provides  the asymptotic estimate of a solution with a collision. 
One modifies the curve around the collision time such that 
the value of the action functional is lower. 
The technique  was used for getting several solutions
(for instance, \cite{Chen06, Shibayama, Venturelli}). 

\item[Averaging] Marchal \cite{Marchal} proved that the minimizers of the fixed ends problem 
is free of collision. 
For a collision path, he proved that 
the average of the value of action functional  for  all direction  near the collision path
is lower the value for the collision path
(See also \cite{Chenciner}).
Ferrario-Terracini developed the technique under the symmetric constraint, 
and provided a criterion  called ``the rotating circle property'' 
that the minimizer has no collision.
\end{description}

In our proof, 
we set the boundary condition 
such that the configuration is rhomboidal at $t=0$ and rectangle at $t=\pi/4$. 
We can avoid the total collision by using the global estimate.
The ``interior'' binary collision (collision which occurs at $t \in (0, \pi/4)$) can be eliminated by applying Ferrario-Terracini's theorem.
But we can not apply the existing methods  to eliminate a ``boundary'' collision (collision which occurs at $t=0, \pi/4$).
 In order to solve this, a new technique is necessary.
By using the scaling technique  established by Tanaka \cite{T93, T94}, 
 we investigate the asymptotic behavior with a minimizer with a boundary collision. 
 Then the behavior induce a contradiction 
 with the fact that the solution minimizes the action functional. 

In order to state our theorem exactly, 
we need to prepare some linear maps.
Let $P_{x}$ and $P_{y}$ be the projection to $x$- and $y$-axis, respectively, 
and let $R_{x}$ and $R_{y}$ be the reflection with respect to $x$- and $y$-axis, respectively.
These are represented by the matrices
\begin{align*}
& & P_{x}&=\left(\begin{array}{cc}1 & 0\end{array}\right), & 
P_{y}&=\left(\begin{array}{cc}0 & 1\end{array}\right), & & \\
& & R_{x}&=\left(\begin{array}{cc}1 & 0 \\0 & -1\end{array}\right), &
 R_{y}&=\left(\begin{array}{cc}-1 & 0 \\0 & 1\end{array}\right).& &
\end{align*}

 The planar four-body problem 
with equal masses is given by the following set of ODEs:
\begin{equation}\label{4bd}
 \frac{d^{2}q_{k}}{dt^{2}}=-\sum_{j \neq k}\frac{q_{k}-q_{j}}{|q_{k}-q_{j}|^{3}}, \qquad q_{k} \in 
 \mathbb{R}^2, \quad  k=1,2,3,4. 
 \end{equation}

Our main result is the following:
\begin{theorem}\label{main}
There is a collisionless $2\pi$-periodic solution $\left(q_{1}\left(t\right), q_{2}\left(t\right), q_{3}\left(t\right), q_{4}\left(t\right)\right)
: \mathbb{R} \to (\mathbb{R}^2)^{4}$
of \eqref{4bd} 
such that
\[
P_{y}q_{1}\left(0\right) > 0, \qquad
P_{x}q_{2}\left(0\right) > 0, \qquad 
P_{x}q_{1}\left(\frac{\pi}{4}\right) > 0, \qquad
P_{y}q_{1}\left(\frac{\pi}{4}\right) < 0, \qquad
\]
and that for any $t \in \mathbb{R}$
 \begin{align}
 &q_{1}\left(t\right)=R_{y}{q_{1}\left(-t\right)}=R_{x}{q_{2}\left(\frac{\pi}{2}-t\right)}, &
q_{2}\left(t\right)&=R_{x}{q_{2}\left(-t\right)} \label{par} \\
&q_{1}\left(t\right)=-q_{3}\left(t\right), & q_{2}\left(t\right)&=-q_{4}\left(t\right). \label{sym}
\end{align}

\end{theorem} 
It easily follows from
\eqref{par} and \eqref{sym} that
\[ q_{1}\left(t\right)=q_{2}\left(t+\frac{\pi}{2}\right)= q_{3}\left(t+\pi\right) =q_{4}\left(t+\frac{3\pi}{2}\right)=q_{1}\left(t+2\pi\right). \]
It means this solution is  a so-called  simple choreography.
The trivial solution (the rotating square) does not satisfy these properties and 
a solution with these properties have not been proved to exist. 
The solution has a same feature as Gerver's numerical solution.

This paper is organized as the following.
In the next section, 
we introduce the variational formulation of the $n$-body problem and some known results.
In Section 3 we 
prove the existence of a generalized periodic solution by using the variational method.
We eliminate the total collision in Section 4 and 
the binary collision in Section 5.
In Appendix A, we provide a detailed argument
for the asymptotic behavior 
 of the binary collision for a certain case by using  the Levi-Civita coordinates.
In Appendix B, we give a numerical result 
confirming that the theoretically obtained solution is identical 
with Gerver's numerical solution.

\section{Symmetry and existence of a minimizer}
We introduce some known results  
which we use in the following sections. 
The $N$-body problem is given by the following set of ODEs: 
\begin{equation}\label{nbd}
 \frac{d^{2}q_{k}}{dt^{2}}=-\sum_{j \neq k}\frac{m_{j}}{|q_{k}-q_{j}|^{3}}(q_{k}-q_{j}), \qquad q_{k} \in V, \quad  k=1,\dots, N
 \end{equation}
where $m_{j}>0$ and $V=\mathbb{R}^{2}$ or $\mathbb{R}^{3}$.
The equation \eqref{nbd} is 
equivalent to the variational problem with respect to the action functional: 
\[ \mathcal{A}(q)=\int_{0}^{T} \frac{1}{2}\sum_{i=1}^{N}m_{i}|\dot{q}_{i}|^{2}+\sum_{1\le i<j \le N} \frac{m_{i}m_{j}}{|q_{i}-q_{j}|}dt. \] 
Let $\mathcal{X}$ be defined by
\[ \mathcal{X}=\left\{ q=(q_{1}, \dots,  q _{N}) \in V^{N} \mid \sum_{i=1}^{N}m_{i}q_{i}=0\right\}\] 
and let 
\[ \Delta_{ij}=\{q \in \mathcal{X} \mid q_{i}=q_{j}\}, \quad \Delta = \cup_{1\le i<j\le 4} \Delta_{ij}.\]
$\mathcal{X} \backslash \Delta$ is denoted by $\hat{\mathcal{X}}$.
$\mathcal{A}$ is defined on the Sobolev space $\Lambda=H^{1}(\mathbb{T}, \mathcal{X})$ of 
$\mathcal{X}$-valued function 
on $\mathbb{T}=\mathbb{R}/2 \pi \mathbb{Z}$.
$\hat{\Lambda} =H^{1}(\mathbb{T},\hat{\mathcal{X}})$ is the subspace of collision-free paths.

We consider an action of a finite group $G$ to $\Lambda$  which has the 
following property: 
there are representations 
\begin{align*}
\tau&: G \to O(2), \\
\rho&: G \to O(\dim V), \\
\sigma&: G \to \mathfrak{S}_{N}
\end{align*}
such that for $g \in G, q(t)=(q_{1}, \dots, q_{N})(t) \in \Lambda$
\[ g \cdot ((q_{1}, \dots , q_{N} ) (t)) = ( \rho(g)q_{\sigma(g^{-1})(1)},
\dots ,\rho(g)q_{\sigma(g^{-1})(N)})(\tau(g^{-1})t) \]
 for $g \in G$ and $q(t)=(q_{1}(t), \dots, q_{N}(t)) \in \Lambda$.
 Let $\Lambda^{G}$ and $\hat{\Lambda}^{G}$ be the set of loops fixed by $G$ in 
 $\Lambda$ and $\hat{\Lambda}$, respectively, 
 and let $\mathcal{A}^{G}=\mathcal{A}|_{\Lambda^{G}}$.
 
 \begin{proposition}[Palais principle \cite{P}] \label{P:Palais}
 If $\mathcal{A}$ is invariant under the group action of $G$, 
 then 
a critical point of  $\mathcal{A}^{G}$ in $\hat{\Lambda}^{G}$  is a critical point of $\mathcal{A}$ in $\Lambda$.
 \end{proposition}
 
 The group $G$ acts on $\mathcal{X}$ by $\rho$ and $\sigma$.
 $\mathcal{X}^{G}$ and $\hat{\mathcal{X}}^{G}$ are defined as the set of points 
 fixed by $G$ in $\mathcal{X}$ and $\hat{\mathcal{X}}$, respectively.
 
 A functional $\mathcal{F}$ on a Banach space is called  coercive, 
 if the value $\mathcal{F}(q)$ of the functional diverges to infinity as the norm $\|q\|$ diverges to infinity.
 
 \begin{proposition}[\cite{FT04}, Proposition 4.1]\label{P:exist}
 If $\mathcal{X}^{G}=\{0\}$,
the action functional $\mathcal{A}^{G}$ is coercive. 
\end{proposition}
  
 For a fixed $t \in \mathbb{T}$, let $G_{t}$ be the isotropy subgroup of $G$ at $t$ under the 
 $\tau$-action, and for 
 $i \in \{1, \dots, N\}$, let $G_{t}^{i}$ be the isotropy subgroup of $G_{t}$ at $i$ under the $\sigma$-action, namely,
 \begin{align*}
 G_{t}&=\{g \in G \mid \tau(g)t=t\} \\
 G_{t}^{i}&=\{ g \in G_{t} \mid \sigma(g)i=i \}. 
 \end{align*} 
 \begin{definition}
 We say a finite group $G$ acts on $\Lambda$ with the rotating circle property (or $G$ has the rotating circle property), 
 if for any $G_{t}$ and for at least $N-1$ indices $i$ there exists a circle in $V$ such that $G_{t}$ acts on the circle by rotation
 and that the circle is contained in $(V)^{G_{t}^{i}}$.
 \end{definition}
 \begin{proposition}[\cite{FT04}, Theorem 10.3]\label{p:ft}
  Let $K=\ker \tau$. 
 Consider a finite group $K$ acting on $\Lambda$ with the rotating circle property.
 Then a minimizer of the $K$-equivariant fixed-ends (Bolza) problem is free of collisions. 
 \end{proposition}
 
 Now we consider the case of choreography constraint. 
  Let $d=\dim V$.
 Assume that  the masses are equal $m_{i}=1 (i=1, \dots, N)$ and 
 take $G$ as  the cyclic group $C_{N}=\langle g \mid g^{N}=1 \rangle$ 
 where 
 \[ \tau(g)=\left(\begin{array}{cc}\cos \frac{2 \pi}{N} & -\sin \frac{2 \pi}{N} \\\sin \frac{2 \pi}{N} & \cos \frac{2 \pi}{N}\end{array}\right), 
 \quad  \rho(g)=\mathrm{Id}_{d}, \quad \sigma(g)=(1~2~\dots~N). \]
 \begin{proposition}[\cite{BT}]\label{prop:BT}
  For every $d\ge 2$, the absolute minimum of $\mathcal{A}^{C_{N}}$ (the action functional under the simple choreography constraint) is attained on a relative equilibrium motion associated 
  with the regular $N$-gon.
 \end{proposition}
 
\section{The existence of the generalized solution}
 
 We consider the planar four-body problem with equal masses:
 \[ V = \mathbb{R}^{2},  \quad N=4, \quad m_{1}=m_{2}=m_{3}=m_{4}=1,\]
 and  take $G$ as
 \[H
= \mathbb{Z}_{2} \times D_{8} =\langle g_{1} \mid g_{1}^{2} =1 \rangle \times
\langle g_{2}, g_{3} \mid g_{2}^{2}=g_{3}^{4}=
 (g_{2}g_{3})^{2}=1\rangle, \]
where 
\begin{align}
\tau(g_{1})&=\mathrm{Id}_{2},& \rho(g_{1})&=- \mathrm{Id}_{2}, &\sigma(g_{1})&=(1 \quad 3)(2 \quad 4), \notag \\
\tau(g_{2})&=\left(\begin{array}{cc}1 & 0 \\0 & -1\end{array}\right),& \rho(g_{2})&=\left(\begin{array}{cc}-1 & 0 \\0 & 1\end{array}\right), &\sigma(g_{2})&=(2 \quad 4), \label{eqn:sym} \\
\tau(g_{3}) &=\left(\begin{array}{cc}0 & -1 \\1 & 0\end{array}\right), & 
\rho(g_{3})&=\mathrm{Id}_{2}, &
\sigma(g_{3})&=(1\quad 2 \quad 3 \quad 4).\notag
\end{align}
This action stands for the symmetry with which 
the super-eight is endowed.
The action by $C_{4}:=\langle g_{3}\rangle$ stands for choreography constraint.
From Proposition  \ref{prop:BT}, 
the minimizer of $\mathcal{A}^{C_{4}}$ is the rotating square.
It is easy to check that 
the action of $H$  satisfies the assumption of Palais principle.
 Since $\Lambda^{H}$ is a subset of $\Lambda^{C_{4}}$ 
 and the rotating square is included in $\Lambda^{H}$, 
the minimizer of  $\mathcal{A}^{H}$ is the rotating square.
In order to obtain the super-eight solution, 
we need to attain another local minimizer.
Define $\Omega$ by
\[\Omega=\left\{ q \in \Lambda^{H} \mid
 P_{y}q_{1}(0) \ge 0, \quad P_{x}q_{2}(0) \ge 0, \quad P_{x} q_{1}(\pi/4) \ge 0, \quad P_{y}q_{1}(\pi/4)  \le 0 \right. \}.  \]
 and let $\hat{\Omega}=\Omega \cap \hat{\Lambda}^{H}$.
 $\hat{\Omega}$ is an open set of $\hat{\Lambda}^{H}$.
 $\Omega$ does not include the rotating square.
If there is  a minimizer of 
 $\mathcal{A}|_{\Omega}$ (not on the boundary of $\Omega$), 
 it is a local minimizer of $\mathcal{A}^{G}$, and safisfies \eqref{par} and \eqref{sym}.
Since $\mathcal{A}|_{\Omega}$ is coersive, 
 there is a minimizer of $\mathcal{A}|_{\Omega}$.
 We need to prove that the minimizer does not belongs to its boundary $\partial \Omega$
 which corresponds to the loops with a collision.
 
 For $q(t)=(q_{1}(t), q_{2}(t), q_{3}(t), q_{4}(t)) \in \Omega$, 
 $\mathcal{A}(q)$ can be written as
\[ \mathcal{A}(q)= 16\int_{0}^{\pi/4} 
\frac{1}{2}( |\dot{q}_{1}|^{2}+|\dot{q}_{2}|^{2} )+
\frac{1}{|q_{1}-q_{2}|}+\frac{1}{|q_{1}+q_{2}|}+\frac{1}{2|q_{1}|}+\frac{1}{2|q_{2}|} dt,\]
since the configuration always holds $q_{1}=-q_{3}, q_{2}=-q_{4}$.
Hence it is sufficient to consider $q_{1}$ and $q_{2}$.
 Let $\cal Y$ be the 
set of configurations without collision:
\[ \mathcal{Y}= \left\{ (q_{1},q_{2}) \in (\mathbb{R}^{2})^{2} \left| q_{1} \neq 0, q_{2} \neq 0, q_{1} \neq q_{2}, 
q_{1} \neq -q_{2}
\right. \right\}.  
\]
Define a set $\Gamma$ and
$\hat{\Gamma}$ by 
\begin{eqnarray*}
\Gamma&=&\left\{ (q_{1}(t), q_{2}(t)) \in H^{1}([0, \pi/4],  (\mathbb{R}^{2})^{2}) \left|
\begin{split} 
P_{x} q_{1}(0)= P_{y} q_{2}(0)=0, \\
 P_{y} q_{1}(0)\ge  0, \quad P_{x} q_{2}(0) \ge 0, \\
R_{y} q_{1}(\pi/4)=q_{2}(\pi/4), \\
P_{x} q_{1}(\pi/4) \ge 0, \quad P_{y} q_{1}(\pi/4) \le 0
\end{split}\right. \right\},  \\
\hat{\Gamma}&=&\Gamma \cap H^{1} ([0, \pi/4], \mathcal{Y}). 
\end{eqnarray*}

Consider the map $\hat{\Omega} \to \hat{\Gamma}$ 
by corresponding $q=(q_{1}(t), q_{2}(t), q_{3}(t), q_{4}(t)) (t \in [0, 2 \pi])$ in $\hat{\Omega}$ 
to $\gamma(t):=(q_{1}(t), q_{2}(t)) (t \in [0, \pi/4])$. 
This map is bijective.
  We can consider the action functional on $\hat{\Gamma}$ defined by 
 \[\mathcal{J}(\gamma)=\int_{0}^{\pi/4} \frac{1}{2}( |\dot{q}_{1}|^{2}+|\dot{q}_{2}|^{2} )+
\frac{1}{|q_{1}-q_{2}|}+\frac{1}{|q_{1}+q_{2}|}+\frac{1}{2|q_{1}|}+\frac{1}{2|q_{2}|}dt\]
 instead of $\mathcal{A}$.
 
 In order to eliminate the possibility of a collision, 
 we 
 add  the strong force part to  the Lagrangian:
\begin{align*}
 L^{\varepsilon}(\gamma, \dot{\gamma}) 
&= 
\frac{1}{2}( |\dot{q}_{1}|^{2}+|\dot{q}_{2}|^{2} )+
\frac{1}{|q_{1}-q_{2}|}+\frac{1}{|q_{1}+q_{2}|}+\frac{1}{2|q_{1}|}+\frac{1}{2|q_{2}|}  \\
&\quad +\frac{\varepsilon}{|q_{1}-q_{2}|^{2}}+\frac{\varepsilon}{|q_{1}+q_{2}|^{2}}+\frac{\varepsilon}{2|q_{1}|^{2}}
+\frac{\varepsilon}{2|q_{2}|^{2}}
 \end{align*}
and  consider  the action functional for the Lagrangian: 
 \begin{align*}
 \mathcal{J}^{\varepsilon}(\gamma)&= \int_{0}^{\pi/4} L^{\varepsilon}(\gamma, \dot{\gamma}) dt. 
\end{align*}
We can easily check $\mathcal{X}^{H}=\{0\}$, and hence 
$\mathcal{A}^{H}$ is coercive. 
Since the restricted functional $\mathcal{A}|_{\Omega}$ is also coercive, 
$\mathcal{J}$ is coercive.
As $\mathcal{J}^{\varepsilon}(\gamma)$ is greater than $\mathcal{J}(\gamma)$ for a fixed  $\varepsilon>0$, 
$\mathcal{J}^{\varepsilon}(\gamma)$ is also coercive.

It is known (for instance see \cite{G77}) that
the value of an action functional with a strong force diverges as the  path converges to a collision path.
Therefore the minimizer $\gamma^{\varepsilon}$ of $\mathcal{J}^{\varepsilon}$ for $\varepsilon >0$ is attained in $\hat{\Gamma}$.
Since 
\[ \mathcal{J}(\gamma^{\varepsilon}) < \mathcal{J}^{\varepsilon}(\gamma^{\varepsilon}) \le \mathcal{J}^{1}(\gamma^{1}),\]
for $0<\varepsilon<1$, 
$\mathcal{J}(\gamma^{\varepsilon}) ~(0<\varepsilon<1)$ is bounded. 
Therefore
the set $\{ \gamma^{\varepsilon} \mid 0< \varepsilon < 1\}$ is bounded in $\Gamma$
because of the coercivity of $\mathcal{J}$. 
There exists a  subsequence $\varepsilon_{n} \searrow 0 (n \to \infty)$ such that $\gamma^{\varepsilon_{n}}$  converges to a point
 $\gamma^{0}$ in $\Gamma$.

\begin{proposition}\label{p:min}
$\mathcal{J}(\gamma^{0}) = \inf_{\gamma \in \Gamma}\mathcal{J}(\gamma)$.
\end{proposition}
\begin{proof}
$\mathcal{J}^{\varepsilon_{n}}(\gamma^{\varepsilon_{n}})$ monotonically decreases as $n$ diverges, 
and hence the limit of the sequence exists.
From Fatou's lemma, we have
 \begin{align*}
  \lim_{n\to \infty} \mathcal{J}^{\varepsilon_{n}}(\gamma^{\varepsilon_{n}}) &=\lim_{n \to \infty} \int_{0}^{\pi/4} L^{\varepsilon_{n}}(\gamma^{ \varepsilon_{n}}, \dot{\gamma}^{\varepsilon_{n}} ) dt \\
  & \ge  \int_{0}^{\pi/4} \liminf_{n \to \infty} L^{\varepsilon_{n}}(\gamma^{ \varepsilon_{n}}, \dot{\gamma}^{\varepsilon_{n}} ) dt \\
  &= \int_{0}^{\pi/4} L^{0}(\gamma^{0}, \dot{\gamma}^{0}) dt \\
  &= \mathcal{J}(\gamma^{0}).
  \end{align*}
  For  any $\gamma \in \hat{\Gamma}$, 
the inequality 
\[\mathcal{J}(\gamma^{0})\le \mathcal{J}^{\varepsilon_{n}}(\gamma^{\varepsilon_{n}}) \le 
\mathcal{J}^{\varepsilon_{n}}(\gamma)\]
satisfies.

On the other hand, 
for arbitrary fixed $\gamma \in \hat{\Gamma}$ and $t$, $L^{\varepsilon_{n}}(\gamma (t), \dot{\gamma}(t))$  monotonically decreases as $n$ increases.
Hence 
\begin{align*}
\lim_{n \to \infty} \mathcal{J}^{\varepsilon_{n}}(\gamma) &= 
\lim_{n \to \infty} \int_{0}^{\pi/4} L^{\varepsilon_{n}} (\gamma, \dot{\gamma}) dt \\
&= 
 \int_{0}^{\pi/4}\lim_{n \to \infty} L^{\varepsilon_{n}} (\gamma, \dot{\gamma}) dt \\
 &= 
 \int_{0}^{\pi/4} L^{0} (\gamma, \dot{\gamma}) dt \\
 &=\mathcal{J}(\gamma) 
\end{align*}
Therefore we get 
\[ \mathcal{J}(\gamma^{0}) \le \lim_{n \to \infty} \mathcal{J}^{\varepsilon_{n}} (\gamma^{\varepsilon_{n}}) 
\le \lim_{n\to \infty} \mathcal{J}^{\varepsilon_{n}} (\gamma) =\mathcal{J}(\gamma).\]
Thus 
\[ \mathcal{J}(\gamma^{0})  \le \inf_{\gamma \in \hat{\Gamma}} \mathcal{J}(\gamma).\]
  
Assume that $\bar{\gamma} \in \partial \Gamma$ minimizes $\mathcal{J}$. 
From the Sundman estimate, 
the asymptotic behavior near collision time $t_{0}$ can be represented by 
\[ \bar{\gamma}(t)=c (t-t_{0})^{2/3} + O(t-t_{0}). \]
where $c \in (\mathbb{R}^{2})^{2} \backslash \mathcal{Y}$.
If $0<t_{0}< \pi/4$, for $d \in (\mathbb{R}^{2})^{2}$ and $\delta t >0$, 
we define 
\[ \chi (t)= \begin{cases} 
0 & (0 \le t < t_{0}-\delta t) \\
d (t-t_{0}+\delta t) & ( t_{0}-\delta t \le t < t_{0} )\\
-d (t-t_{0}-\delta t) & (t_{0}\le t < t_{0} +\delta t )\\
0 &  (t_{0} +\delta t  \le t \le \pi/4).
\end{cases}\]
We similarly define $\chi$ in the case of  $t_{0}=0$ or $\pi/4$. 
Take $d$ such that 
$\bar\gamma(t_{0})+\chi(t_{0}) \in \mathcal{Y}$ for small $\delta t>0$.
From easy computation, it follows that 
\[ \mathcal{J}(\bar{\gamma}+\chi)- \mathcal{J}(\bar{\gamma}) =O((\delta t)^{2/3}).\]
Such an estimate have been gotten when one used the local estimate. 
The collision times are isolated 
and we can make the same  modification at all collision time.
The functional $\mathcal{J}$ can have as close  value to $ \mathcal{J}(\bar{\gamma}) $  in $\hat{\Gamma}$ 
 as one need.
This follows that 
\[ \inf_{\gamma \in \hat {\Gamma}} \mathcal{J}(\gamma)=\inf_{\gamma \in \Gamma} \mathcal{J}(\gamma).\]

  \end{proof} 
  
  Thus  
  if $\gamma^{0} \in \hat{\Gamma}$, $\gamma^{0}$ is a minimizer of $\mathcal{J}$, 
  and hence $\gamma^{0}$ is a solution.
We will prove that $\gamma \in \hat{\Gamma}$ in the following two sections.

\section{Elimination of  total collisions}
\begin{proposition}\label{pr:totcol}
$\gamma^{0}$ has no total collision, 
i.e. for any $t$, $\gamma^{0}(t) \neq 0$.
\end{proposition}
\begin{lemma}\label{lem:tot}
For any  $\gamma \in \Gamma$ with  a total collision,  
 $\mathcal{J}(\gamma)$ is greater than 9.
\end{lemma}

\begin{proof} 

We first consider the collinear Kepler problem: 
\[ \frac{d^{2}\xi}{dt^{2}}=-\alpha \xi^{-2} \qquad \xi \in \mathbb{R}\]
where $\alpha>0$ is a constant.
For any $T_{0}>0$, there is a unique solution such that $\xi_{\alpha, T_{0}}(0)=\dot{\xi}_{\alpha, T_{0}}(T_{0})=0$ and 
that $\dot \xi_{\alpha, T_{0}}(t) > 0$ for any $t \in (0, T_{0})$.
This is degenerate (as the eccentricity converges to 1)  Kepler motion with period $2T_{0}$.
It is known \cite{G77} that 
$\xi_{\alpha, T_{0}}$ minimizes the action functional for the collinear Kepler motion
\[ \int_{0}^{T_{0}} \frac{1}{2} \dot{\xi}^{2}+\frac{\alpha}{\xi} dt \] 
on $\{ \gamma \in H^{1}([0, T_{0}], \mathbb{R}) \mid \gamma(t) =0 \mathrm{~for~some~} t \in [0, T_{0}]\}$ 
and that the minimum value is 
\begin{equation}\label{valkep}
 2^{-1} \cdot 3\pi^{2/3} \alpha^{2/3}T_{0}^{1/3}.
 \end{equation}

Now we move back to the variational problem of $\mathcal{J}$.
Assume that $\gamma \in \Gamma$ has a total collision at $t_{0} \in [0, \pi/4]$.
Denote $\gamma(t)=(q_{1}(t), q_{2}(t))=r(t)(s_{1}(t), s_{2}(t))$
where 
$r(t)=\sqrt{|q_{1}(t)|^{2}+|q_{2}(t)|^{2}}$.
It is easily seen that the minimum value of $U(s_{1}, s_{2})$ on $\{(s_{1}, s_{2}) \in (\mathbb{R}^2)^{2} \mid 
|s_{1}|^{2}+|s_{2}|^{2}=1\}$ is  $4+\sqrt{2}$
and that it is attained at regular square configuration($|s_{1}|=|s_{2}|, s_{1} \perp s_{2}$).
The minimizing orbit  with a total collision is homothetic orbit with this configuration(Figure \ref{fig:tocol}).
\begin{figure}
\begin{center}
\includegraphics[width=2in]{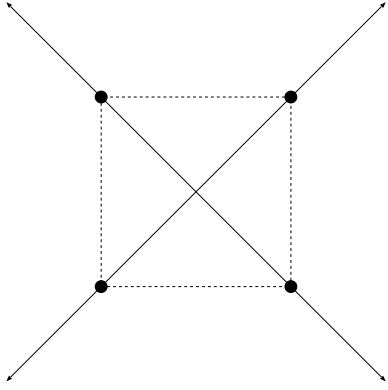}
\caption{Minimizing total collision orbit}
\label{fig:tocol}
\end{center}
\end{figure}
Since $r(t_{0})=0$ for some $t_{0} \in [0, \pi/4]$, it  follows from  \eqref{valkep} that 
\begin{align*}
\mathcal{J}(\gamma^{0}) &=\int_{0}^{\pi/4} \frac{1}{2} (\dot{r}^{2} +r^{2} (|\dot{s}_{1}|^{2}+|\dot{s}_{2}|^{2})) + r^{-1} U(s_{1}, s_{2}) dt \\
& \ge \int_{0}^{\pi/4} \frac{1}{2} \dot{r}^{2} + (4+\sqrt{2}) r^{-1}  dt \\
&\ge 2^{-4/3}\cdot 3 (1+2\sqrt{2})^{2/3} \pi \\
&\approx 9.153.  
\end{align*} 
\end{proof}

 \begin{lemma} \label{lem:testps}
 \begin{equation}
\label{esttest}
\inf_{\hat{\Gamma}}  \mathcal{J} < 5. 
  \end{equation}
 \end{lemma}
 \begin{proof}
 
 We define a test path $\gamma_{\mathrm{test}}$
 as follows: 
 \begin{align*}
 q_{1}(t) &= ( t, \frac{\pi}{4}- 2t) \\
 q_{2}(t)&=(\frac{\pi}{2}-t, t ). 
 \end{align*}
 See Figure \ref{fig:testp}.
 \begin{figure}
\begin{center}
\includegraphics[width=2in]{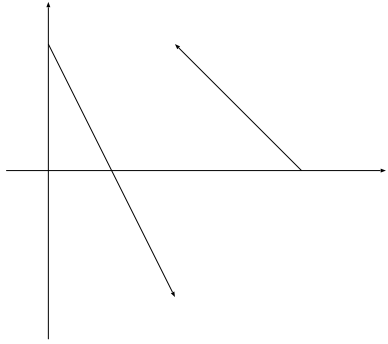}
\caption{Test path}
\end{center}
\end{figure}
 
 From an easy computation, the kinetic part for the test path has a constant value  along this path:
 \[ \frac{1}{2} (|\dot{q}_{1}(t)|^{2}+|\dot{q}_{2}(t)|^{2})\equiv \frac{7}{2}.\]
 Next we estimate the potential part.
 Since
 \[ |q_{1}|=\sqrt{t^{2}+(\frac{\pi}{4}- 2t)^{2} }
 =\sqrt{5(t-\frac{\pi}{10})^{2}+\frac{\pi^{2}}{80}}, \]
 we have
\[ \frac{1}{4|q_{1}|} \ge\frac{1}{4} \sqrt{\frac{80}{ \pi^{2}}} = \frac{\sqrt{5}}{ \pi}. \]
Hence 
\[ \int_{0}^{\pi/4} \frac{1}{4|q_{1}|}  dt \ge \frac{\sqrt{5}}{ \pi} \frac{\pi}{4}  =\frac{\sqrt{5}}{4}. \]
Similarly we get 
\[ \int_{0}^{\pi/4}  \frac{1}{4|q_{2}|}  \ge  \frac{\sqrt{2}}{8}, \quad 
\int_{0}^{\pi/4} \frac{1}{|q_{1}-q_{2}|}  \le \frac{\sqrt{13}}{4}, \quad 
\int_{0}^{\pi/4} \frac{1}{|q_{1}+q_{2}|} \le  \frac{1}{2}. \]
Consequently we have the following estimate:
  \begin{align*}
 \mathcal{J}(\gamma_{\mathrm{test}}) \le \frac{7 \pi}{8}+ \frac{\sqrt{5}}{4}+\frac{\sqrt{2}}{8}+\frac{\sqrt{13}}{4}+\frac{1}{2} 
 \approx  4.886.
 \end{align*}
 \end{proof}
 \begin{proof}[Proof of Proposition \ref{pr:totcol}]
Above two lemmata and Proposition \ref{p:min}, 
$\gamma^{0}$ has no total collision.
\end{proof}

\section{Elimination of binary collision}
\begin{proposition}\label{pr:bcol}
 $\gamma^{0}$ has no binary collision.
 \end{proposition}
The proof varies according to the collision time.

\paragraph{Elimination of binary collision at $t=0$}
The binary collisions are classified into four types with respect to the configuration:
\begin{align*}
\mathrm{I}&: q_{1}=0, q_{2} \neq 0\\
\mathrm{II}&: q_{2}=0, q_{1} \neq 0\\
\mathrm{III}&: q_{1}=q_{2} \neq 0, \\
\mathrm{IV}&: q_{1}=-q_{2} \neq 0.
\end{align*}
The binary collisions with type III and IV at $t=0$ are total collision
because of the symmetry. 
As we proved in the previous section, these types do not occur.

We consider the type I. 
We investigate the behavior by using Tanaka's technique \cite{T93}. 
We can represent the action functional  by
\[ \mathcal{J}^{\varepsilon}(\gamma^{\varepsilon})
=\int_{0}^{\pi/4} \frac{1}{2} |\dot{q_{1}^{\varepsilon}}|^{2}
+\frac{1}{2|q_{1}^{\varepsilon}|} +\frac{ \varepsilon }{2|q_{1}^{\varepsilon}|^{2}}
+f^{\varepsilon}(\dot{q^{\varepsilon}}_{2}, q_{1}^{\varepsilon}, q_{2}^{\varepsilon})dt \] 
where  
\[f^{\varepsilon}(\dot{q}_{2}, q_{1}, q_{2})=\frac{1}{2} |\dot{q_{2}}|^{2}+\frac{1}{|q_{1}-q_{2}|} 
+\frac{1}{|q_{1}+q_{2}|} + \frac{1}{2|q_{2}|} 
+\frac{\varepsilon}{|q_{1}-q_{2}|^{2}} 
+\frac{\varepsilon}{|q_{1}+q_{2}|^{2}}+\frac{ \varepsilon }{2|q_{2}|^{2}}. \] 
Let $\gamma^{\varepsilon}(t)=(q_{1}^{\varepsilon}(t), q_{2}^{\varepsilon}(t))$. 
Let $q_{1}^{0}(0)=0$. 
We define $\delta_{n}$ by 
\[ \delta_{n}=|q_{1}^{\varepsilon_n}(0)| >0 \]
and 
define a scale transformation $x_{n}$ of 
$q_{1}^{\varepsilon_n}$ by
\begin{equation*}
x_{n}(s)=\delta_{n}^{-1} q_{1}^{\varepsilon_n}(\delta_{n}^{3/2} s).
\end{equation*}
Taking a subsequence of $n$ if necessary, 
we can assume that 
\begin{equation}\label{eqn:6.36}
\frac{\varepsilon_{n}}{\delta_{n}} \to d \in [0, \infty].
\end{equation}
We first consider the case of $0 \le d < \infty$.
From Tanaka's argument \cite{T93}, 
for any $l>0$, 
$x_{n}$ converges uniformly on $[0, l]$ to the solution $y_{d}$ of the following equations:
\begin{align*}
    &  \ddot{y}+\frac{ y}{2|y|^{3}} + \frac{2dy}{|y|^{4}} =0  \\
    &  y(0)=(0,1) \\
    & \dot{y}(0)=\pm\sqrt{2(1+d)}(1, 0).
\end{align*}

Define $r_{d}(s)$ and $\theta_{d}(s)$ by
\[ y_{d}(s)=r_{d}(s) (\cos \theta_{d}(s), \sin \theta_{d}(s)), \quad \theta_{d}(0)=0. \]
As $s$  goes to infinity, $r_{d}(s)$ diverses to infinity and 
 $\theta_{d}(s)$ converges   to $\frac{\pi}{2}\mp \pi \sqrt{1+d}$ (See Figure \ref{fig:scaling}).
 \begin{figure}
\begin{center}
\includegraphics[width=2in]{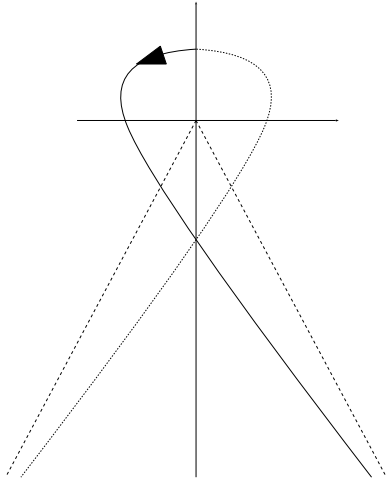}
\caption{The behavior of $y_{d}$}
\label{fig:scaling}
\end{center}
\end{figure}

\begin{lemma}
$y_{d}(t) (t \in [0, l])$ is the minimizer of the fixed-ends problem of $y_{d}(t)$
with respect to 
the action functional 
\[ \mathcal{I}_{d}(y)= \int_{0}^{l} \frac{1}{2} |\dot{y}|^{2} +\frac{1}{2 |y|} + \frac{d }{2|y|^{2}}dt.   \]
\end{lemma}
\begin{proof}
$\mathcal{J}^{\varepsilon_{n}}(\gamma^{\varepsilon_{n}})$ can be written by 
\begin{align*}
\mathcal{J}^{\varepsilon_{n}}(\gamma^{\varepsilon_{n}})&=
\delta^{1/2}_{n} \mathcal{I}_{\varepsilon_{n} \delta_{n}^{-1}} (x_{n})
+\int_{0}^{\delta_{n}^{3/2} l} f^{\varepsilon_{n}}(\dot{q}_{2}, q_{1}, q_{2}) dt 
+ \int_{\delta_{n}^{3/2}l }^{  \pi/2} L^{\varepsilon_{n}}(\gamma, \dot{\gamma}) dt.
\end{align*}
Assume that there exists $y^{\ast}$ such that  
\begin{equation}\label{test1}
 \mathcal{I}_{d}(y^{\ast}) < \mathcal{I}_{d}(y_{d}) 
 \end{equation}
and that $y^{\ast}$ has the same fixed-ends: 
\begin{equation}\label{test2} y^{\ast}(0)=y_{d}(0), y^{\ast}(l)=y_{d}(l). 
\end{equation}
We will define a sequence $y_{k}^{\ast}$ converging to $y^{\ast}$ in the Sobolev space
such that 
\[ y_{k}^{\ast}(0)=x_{k}(0), y_{k}^{\ast}(l)=x_{k}(l)\]
as follows.
Note that 
\[ y^{\ast}(0) = x_{n}(0)=(1, 0), \quad y^{\ast}(l) - x_{n}(l) \to 0. \] 
We take a subsequence $n_{k}$, 
such that 
\[  \left|y^{\ast}(l) - x_{n_{k}}(l) \right| \le \frac{1}{k}. \]

$y_{k}^{\ast}$ is defined by 
\[ y_{k}^{\ast}(s)= 
\begin{cases}
y^{\ast}(s) & s \in [0, l-\frac{1}{k}] \\
k(y^{\ast}(l-\frac{1}{k}) - x_{n_{k}}(l)) (l-s) +x_{n_{k}}(l) & s \in [l-\frac{1}{k}, l]. 
\end{cases}
\]
$\mathcal{I}_{d} (y_{k}^{\ast})$ converges to $\mathcal{I}_{d}(y^{\ast} ).$
See Figure \ref{fig:connect}.
\begin{figure}
\begin{center}
\includegraphics[width=2in]{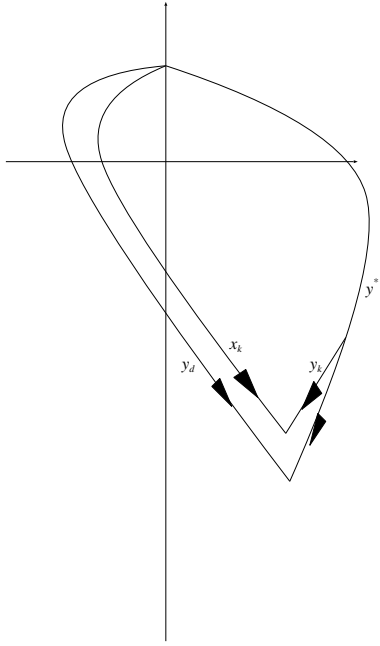}
\caption{$y_{d}, x_{k}, y^{\ast}$ and $y_{k}$.}
\label{fig:connect}
\end{center}
\end{figure}
The difference of the kinetic part of $\mathcal{I}_{d}$ is 
\begin{align*}
&\frac{1}{2}\int_{0}^{l}  \left|\frac{d y_{k}^{\ast}}{ds} \left(s\right)\right|^{2} -
\left|\frac{d y^{\ast}}{ds} \left(s\right)\right|^{2}ds \\
 & =
\frac{1}{2} \int_{l-k^{-1}}^{l}  k^{2}\left(y^{\ast}\left(l-\frac{1}{k}\right) - x_{n_{k}}\left(0\right)\right)^{2}+O(1)ds \\
& =\frac{1}{2} k\left(\frac{d y^{\ast}}{d s}\left(l\right)\frac{1}{k}+o\left(k^{-1}\right)+y^{\ast}\left(l\right) - x_{n_{k}}\left(l\right)\right)^{2} + O(k^{-1})\\
& =k\left(O\left(k^{-1}\right)\right)^{2}+O(k^{-1}) \to 0 \quad \left(k \to \infty\right), 
\end{align*}
and 
the potential function along $y_{k}^{\ast}$ uniformly converges to one along $y^{\ast}$. 

We define $\gamma^{\ast}_{k}(t)=(q^{\ast}_{1}, q^{\ast}_{2})$ by 
\begin{align*}
 q_{1}^{\ast} (t) &=
 \begin{cases}
 \delta_{n_{k}}y_{k}^{\ast}(\delta_{n_{k}}^{-3/2}t) & 0 \le t \le  \delta_{n_{k}}^{3/2}l \\
 q_{1}(t) &\delta_{n_{k}}^{3/2}l \le  t \le \pi/4
 \end{cases}
  \\ 
 q_{2}^{\ast} (t) &=q_{2}(t).
 \end{align*}
Note that 
\[ \max_{0\le t \le \pi/4}|f_{\varepsilon_{n_{k}}}(\dot{q}_{2}, q_{1}, q_{2})-
f_{\varepsilon_{n_{k}}}(\dot{q}_{2}^{\ast}, q_{1}^{\ast}, q_{2}^{\ast})| \le M \delta_{n_{k}}\]
for some constant $M$.
Therefore 
\begin{align*}
&\int_{0}^{\delta_{n_{k}}^{3/2} l} |f_{\varepsilon_{n_{k}}}(\dot{q}_{2}, q_{1}, q_{2})-
f_{\varepsilon_{n_{k}}}(\dot{q}_{2}^{\ast}, q_{1}^{\ast}, q_{2}^{\ast})| dt \le \int_{0}^{\delta_{n_{k}}^{3/2} l} M\delta_{n} dt = l M \delta_{n_{k}}^{5/2}. 
\end{align*}
Consequently 
we get 
\begin{align*}
 \mathcal{J}^{\varepsilon_{n_{k}}}(\gamma^{\ast}_{n_{k}}) -\mathcal{J}^{\varepsilon_{n_{k}}}(\gamma^{0})= 
 \delta_{n_{k}}^{1/2}(\mathcal{I}_{d}(y_{n_{k}}^{\ast})- \mathcal{I}_{d}(y_{d})) +
 O(\delta_{n_{k}}^{5/2}) \\
 = \delta_{n_{k}}^{1/2} (\mathcal{I}_{d}(y^{\ast})- \mathcal{I}_{d}(y_{d})) 
+o(\delta_{n_{k}}^{1/2}).
\end{align*}
For sufficiently large $k$, 
\[ \mathcal{J}^{\varepsilon_{n_{k}}}(\gamma^{\ast}_{n_{k}})<\mathcal{J}^{\varepsilon_{n_{k}}}(\gamma^{\varepsilon_{n_{k}}}).\]
This contradicts the fact that $\gamma^{\varepsilon_{n}}$ is the minimizer.
\end{proof}

Now we consider the case of  $0<d  <\infty$.
We use a similar technique as Coti-Zerati's one \cite{CZ}.
Fix large $l>0$.
There is $0<s_{0}<l$ such that 
$y(s_{0})$ belongs to $y$-axis. 
Consider the reflected curve 
\[ \hat{y}=
\begin{cases}
R_{y} y_{d}(s) & 0 \le s \le s_{0} \\ 
y_{d}(s) & s_{0} \le s. 
\end{cases}
\]
See Figure \ref{fig:reflect2}.
\begin{figure}
\begin{center}
\includegraphics[width=2in]{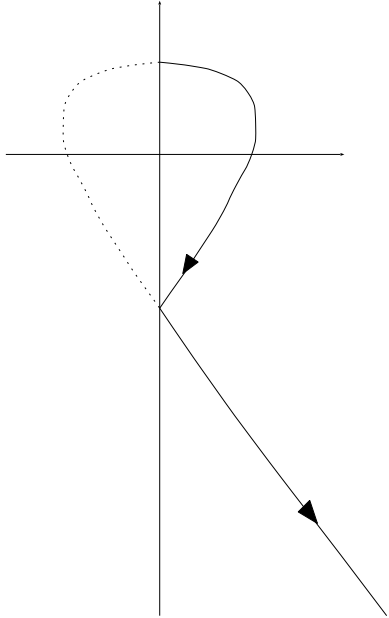}
\caption{The reflected path}
\label{fig:reflect2}
\end{center}
\end{figure}
$\hat{y}$ also has the same value of $\mathcal{I}_{d}$.
This  is also a minimizer of $\mathcal{I}_{d}$ 
but it cannot be smooth at $s=s_{0}$, 
which is a contradiction, 
since any minimizer must be smooth.

Next we consider the case of $d=0$.
Tanaka \cite{T94} also showed the the limit of the velocity direction
is same as $\lim_{s \to \infty} \frac{y_{d}(s)}{|y_{d}(s)|}$. Therefore this case implies
\[ \lim_{t \to +0} \frac{\dot{q}_{1}(t)}{|\dot{q}_{1}(t)|}=(0, -1).\]
If $\dot{q}_{2}(0)=0$, 
$q_{2}(t)$ moves on $x$-axis and $q_{1}(t)$ moves along $y$-axis. 
The total collision occurs at $t = \pi/4$ because of the symmetry.

Let $D_{i}$ be the $i$-th quadrant. 
For example, 
\[ D_{1}=\{(x, y) \in \mathbb{R}^{2} \mid x \ge 0, y \ge 0\}.\]

If the $y$-component of $\dot{q}_{2}(0)$ is negative, 
$q_{1}(t)$ goes into the fourth quadrant for small $t>0$.
It is intuitively clear.
As a rigorous argument, 
we can easily  show this by regularizing the binary collision with the Levi-Civita coordinates 
and by investigating the asymptotic behavior for small $t>0$.
This detail will be written in Appendix A.
As the result, 
if $\dot{q}_{2}(0) \neq 0$, two particles 
must move into a same quadrant for small $t >0$. 

We modify $q_{1}(t), q_{2}(t)$ as these particle move in the separate quadrants:
\begin{align*}
\gamma^{\ast}(t)&=(q_{1}^{\ast}(t), q_{2}^{\ast}(t))\\
q_{1}^{\ast}(t)&= 
\begin{cases}
R_{x}q_{1}(t) & q_{1}(t) \in D_{1} \\
-q_{1}(t) & q_{1}(t) \in D_{2} \\
R_{y}q_{1}(t) & q_{1}(t) \in D_{3} \\
q_{1}(t) & q_{1}(t) \in D_{4} 
\end{cases} \\
q_{2}^{\ast}(t)&= 
\begin{cases}
q_{2}(t) & q_{2}(t) \in D_{1} \\
R_{y} q_{2}(t) & q_{2}(t) \in D_{2} \\
-q_{2}(t) & q_{2}(t) \in D_{3} \\
R_{x}q_{2}(t) & q_{2}(t) \in D_{4}.
\end{cases} 
\end{align*}
See Figure \ref{fig:reflect}.
\begin{figure}
\begin{center}
\includegraphics[width=3in]{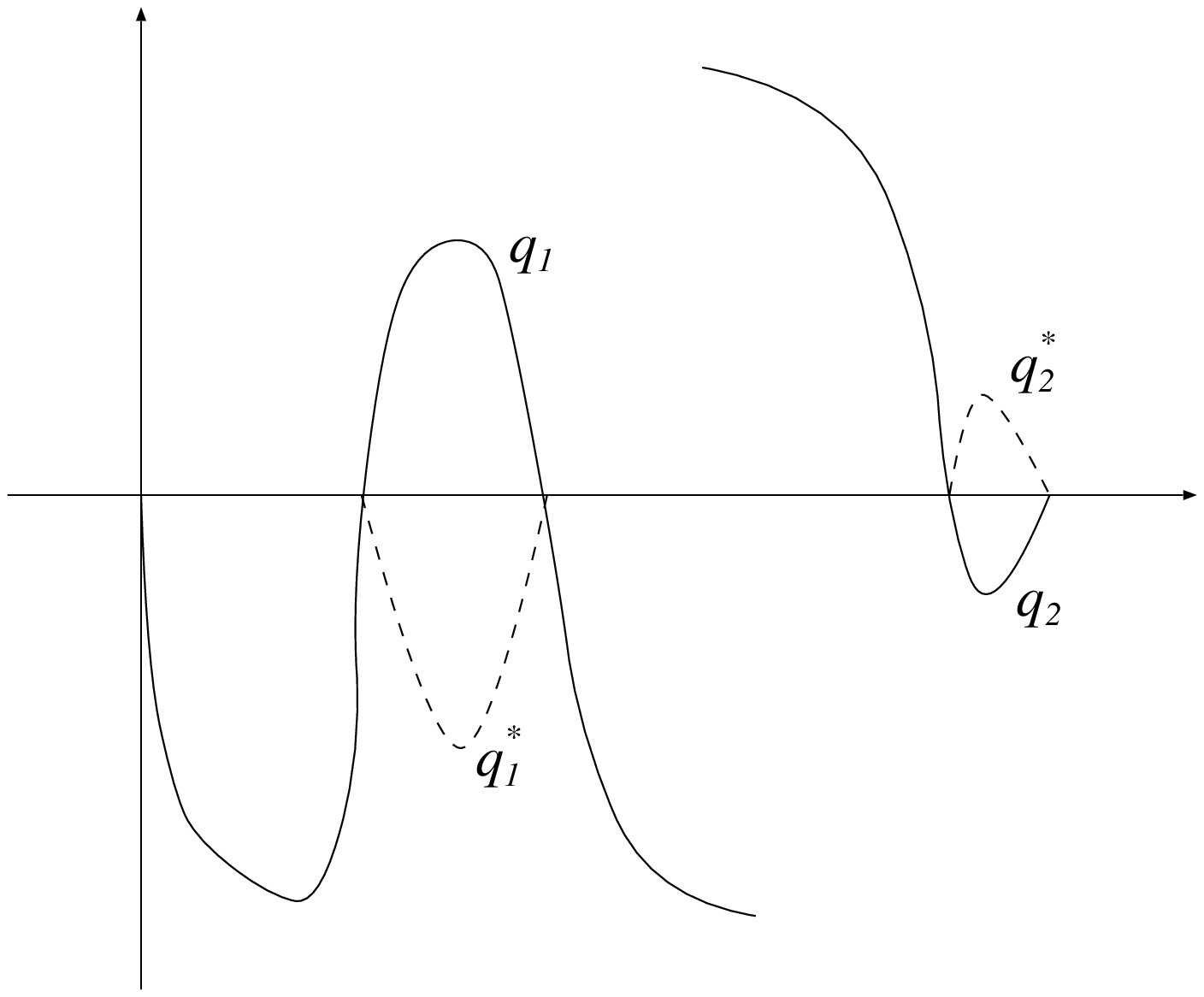}
\caption{$q$ and $q^{\ast}$}
\label{fig:reflect}
\end{center}
\end{figure}
We have
\[ U(q(t)) \ge U(q^{\ast}(t)) \quad \mathrm{and} \quad K(\dot{q})=K(\dot{q}^{\ast})\]
almost everywhere.
If two particles belongs to same quadrant at $t$, 
the inequality \[ U(\gamma(t)) > U(\gamma^{\ast}(t))\]
satisfies.
For small $t>0$, 
two particles belongs to same quadrant.
Therefore we have
\[ \mathcal{A}(\gamma) > \mathcal{A}(\gamma^{\ast}) \]
 which is a contradiction.

In the case of $d=\infty$, 
let 
\[ z_{n}(s)=\delta_{n}^{-1} q_{1}^{\varepsilon_n}\left( \left(\frac{\varepsilon_{n}}{\delta_{n}^{4}}\right)^{-1/2}s\right). \]
For any $l$, $z_{n}$ converges uniformly on $[0, l]$ to the solution $y_{\infty}$ of the following equations:

\begin{align*}
&\ddot{y}+\frac{2y}{|y|^{4}}=0 \\
&y(0)=(0, 1) \\
&\dot{y}(0)=(\pm \sqrt{2}, 0).
\end{align*}
We explicitly denote the solution as follows: 
\[ y_{\infty}(s)=(\cos \sqrt{2}s, \pm \sin \sqrt{2}s).\]
In this case, we can similarly induce a contradiction 
as the case of $0 < d < \infty$
since $y_{\infty}$ has intersections with $y$-axis.

We can eliminate the collision $q_{2}(0)=0$ similarly since 
the situation is essentially same.

\paragraph{Elimination of binary collision at $0< t < \pi/4$.}
Binary collisions at $0 < t < \pi/4$ can also be eliminated by using Tanaka's method. 
But one can immediately eliminate such a  collision  by using  Ferrario-Terracini theorem.
In our setting, 
the kernel of $\tau$ is the subgroup 
generated by $g_{1}$:
\[K:= \ker \tau = \langle g_{1} \rangle\]
where the group was introduced in \eqref{eqn:sym}. 
It is easy to check 
\[ K_{t}= K, \quad K_{t}^{i}=\{1\}.\]
Thus
$K=\ker \tau$ satisfies the rotating circle property
and $\gamma^{0}$ is a minimizer of the fixed-ends problem.
Hence Theorem \ref{p:ft} implies that 
 $q(t)$ has no collision in $(0, \pi/4)$.

\paragraph{Elimination of binary collision at $t= \pi/4$}
Use the coordinates $(Q_{1}, Q_{2}) \in (\mathbb{R}^{2})^{2}$ where
\[ Q_{1}=\frac{1}{\sqrt{2}}(q_{1}+q_{2}), \quad Q_{2}=\frac{1}{\sqrt{2}}(q_{1}-q_{2}).\]
$\mathcal{J}$ can be written as 
\[ \mathcal{J}(q)=\int_{0}^{\pi/2} \frac{1}{2} (|\dot{Q}_{1}|^{2}+|\dot{Q}_{2}|^{2})  
+\frac{1}{\sqrt{2}|Q_{1}+Q_{2}|} + \frac{1}{\sqrt{2}|Q_{1}-Q_{2}|} +
\frac{1}{\sqrt{2}|Q_{1}|} + \frac{1}{\sqrt{2}|Q_{2}|} dt. 
 \]
 The boundary condition at $t=\pi/4$ is 
 \[ P_{x}Q_{1}=0, \quad P_{y}Q_{2}=0, \quad P_{y}Q_{1} >0, \quad P_{x}Q_{2} >0. \]
This situation written by $Q_{1}, Q_{2}$ at $t =\pi/4$ 
is same as one 
written by $q_{1}, q_{2}$ at $t =0$.
Therefore we can eliminate the collision at $\pi/4$ by using 
completely same argument.  

This completes Proposition \ref{pr:bcol} and 
hence Main Theorem.
\section*{Appendix A: Levi-Civita regularization}

Here we identify $\mathbb{R}^{2}$ with $\mathbb{C}$.
The equations for the parallelogram four-body problem are represented by
\[\left\{\begin{split}
\frac{d q_{1}}{dt }&=p_{1} \\
\frac{d q_{2}}{dt}&=p_{2} \\
\frac{dp_{1}}{dt}&=-\frac{q_{1}}{2|q_{1}|^{3}} - \frac{q_{1}-q_{2}}{|q_{1}-q_{2}|^{3}} - \frac{q_{1}+q_{2}}{|q_{1}+q_{2}|^{3}}  \\
\frac{dp_{2}}{dt}&=-\frac{q_{2}}{2|q_{2}|^{3}} - \frac{q_{2}-q_{1}}{|q_{2}-q_{1}|^{3}} - \frac{q_{2}+q_{1}}{|q_{2}+q_{1}|^{3}} \end{split}\right. \qquad (q_{1}, q_{2}, p_{1}, p_{2} \in \mathbb{C}).
\]
We need to regularize the singularity $q_{1}=0 (q_{2} \neq 0)$.
Levi-Civita coordinates $(z, w) \in \mathbb{C}^{2}$are defined by
\[ q_{1}=-\frac{i}{2} z^{2}, p_{1}=-\frac{i w}{\bar{z}}\]
and the time variable $t$ is changed to $\tau$ according to 
$dt = |z|^{2}d\tau$. 
Note that we should identify 
$(z,w )$ with $-(z,w)$.
Let $E$ be the total energy:
\begin{align*}
 E&=\frac{1}{2}(|p_{1}|^{2}+|p_{2}|^{2})-\frac{1}{|q_{1}-q_{2}|}-\frac{1}{|q_{1}+q_{2}|}-
\frac{1}{2|q_{1}|}-\frac{1}{2|q_{2}|}\\
&=\frac{|w|^{2}-2}{2|z|^{2}}+\frac{1}{2}|p_{2}|^{2}-\frac{2}{|z^{2}-2q_{2}|}-\frac{2}{|z^{2}+2q_{2}|}-\frac{1}{2|q_{2}|}.
\end{align*}
Then the equations become
\begin{equation}\label{eqn:LC}
\left\{\begin{split}
\frac{d z}{d\tau }&=w \\
\frac{d q_{2}}{d\tau}&=|z|^{2}p_{2} \\
\frac{dw}{d \tau}
&=z \left(2E-|p_{2}|^{2} +\frac{4}{|z^{2}-2iq_{2}|}+\frac{4}{|z^{2}+2iq_{2}|}+\frac{1}{|q_{2}|}  \right) -4|z|^{2}\bar{z}\left( \frac{z^{2} -2i q_{2}}{|z^{2}-2iq_{2}|^{3}} + \frac{z^{2}+2iq_{2}}{|z^{2}+2iq_{2}|^{3}} \right)\\  
\frac{dp_{2}}{d\tau}&=|z|^{2}\left(-\frac{q_{2}}{2|q_{2}|^{3}} - \frac{4(2iq_{2}-z^{2})}{|2iq_{2}-z^{2}|^{3}} - \frac{4(2iq_{2}+z^{2})}{|2iq_{2}+z^{2}|^{3}} \right).\end{split}\right. 
\end{equation}
Note that the equations are real-analytic at $z=0, q_{2} \neq 0$.

Now consider the solution with the binary collision at $t=0$ and the case of $d=0$.
From $q_{1}(+0)=0, \frac{q_{1}}{|q_{1}|}(+0)=\frac{p_{1}}{|p_{1}|}(+0)=-i$ and the energy relation, 
we have
\[ z(0)=0,  \quad w(0)=\sqrt{2}.\]
If $p_{2}(0)=0$, the complex conjugate $\bar{z}(\tau)$ also satisfies the equation. 
From the unicity of the solution, $z(\tau)=\bar{z}(\tau)$.
Thus $z(\tau) \in \mathbb{R}$.
 
 Next we consider the case of $p_{2}(0) \neq 0$.
From the theory of analytic differential equations, 
the solution is real-analytic at $\tau=0$. 
We denote the Taylor extension for $z, q_{2}$ by 
\[ z=\sum_{k=0}^{\infty}a_{k} \tau^{k}, \quad q_{2}=\sum_{k=0}^{\infty} b_{k}\tau^{k}.\]
From  $z=\sqrt{2}\tau +o(\tau)$, 
we get $t=\frac{2}{3}\tau^{3} +o(\tau^{3})$.
Since $q_{2}(t)$ can be represented by 
$q_{2}(t)=q_{2}(0)+\dot{q}_{2}(0) t +o(t)$, 
the Taylor extension is
\[ q_{2}=b_{0}+b_{3}\tau^{3} + o(\tau^{3})\]
where $b_{3}=\frac{2}{3} \dot{q}_{2}(0).$

We substitute these Taylor extension of $z, q_{2}$ to the third equation 
of \eqref{eqn:LC}.
From the straight forward computation, 
we get 
\[ a_{1}, a_{2}, \dots, a_{9} \in \mathbb{R}\]
and 
\[ \mathrm{Im} ~a_{10}=\frac{2 \sqrt{2}}{15} |b_{0}|^{-5}  b_{0}    (b_{3}i)    \] 

Therefore if $\mathrm{Im} ~ b_{3} >0$ (and $<0$ resp.), 
$\mathrm{Im} ~a_{10}<0$ (and $>0$ resp.).
Therefore $\mathrm{Im} ~q_{1}$ is negative (and positive resp.) for small $t>0$.
Consequently 
$q_{1}$ and $q_{4}$ (and $q_{2}$ resp.) belongs to same quadrant.  

\section*{Appendix B: Numerical result}

We numerically get the minimizer of $\mathcal{J}$ by using the steepest decent method.
 
  The gradient vector $\nabla \mathcal{J}$ of $\mathcal{J}$ at $\gamma \in \hat{\Gamma}$ is defined by 
 \[ \mathcal{J}'(\gamma) \delta = (\nabla \mathcal{J}(\gamma), \delta)_{H^{1}}.\]
Consider the differential equation 
 \begin{equation}\label{Fin}
  \frac{d \gamma}{d s}= -\nabla \mathcal{J}(\gamma) \qquad \gamma \in \Gamma.
  \end{equation} 
 Let  $\mathbf{a}_{1}=(1, 0, 0, 0), \mathbf{a}_{2}=(0, 1, 0, 0),  \mathbf{a}_{3}=(0, 0,1, 0),  \mathbf{a}_{4}=(0,0,0,1)$.
 
The set 
\[ \left\{ \frac{ \sin l t }{\sqrt{\pi (1+l^{2})}}\mathbf{a}_{1}, \frac{\cos l t }{\sqrt{\pi(1+l^{2})}} \mathbf{a}_{2} , 
  \frac{ \cos l t }{\sqrt{\pi (1+l^{2})}}\mathbf{a}_{3}, \frac{\sin l t }{\sqrt{\pi(1+l^{2})}} \mathbf{a}_{4}  \mid l \in \mathbb{N}\right\}\]
   is an orthonormal  basis of the tangent space $T_{\gamma}\hat{\Gamma}=:T$ of $\hat{\Gamma}$.  
   We denote $q \in \hat{\Gamma}$ by 
   \[q= \sum_{l=1}^{\infty} \sum_{j=1}^{4} \left(\xi_{lj}\frac{ \sin l t }{\sqrt{\pi (1+l^{2})}}\mathbf{a}_{j}+ \eta_{lj}\frac{\cos l t }{\sqrt{\pi(1+l^{2})}} \mathbf{a}_{j}\right) \in \hat{\Gamma}. \]
   The differential equation \eqref{Fin} is approximately represented by 
\begin{align}
\frac{d \xi_{lj}}{d s}&= - \left(\nabla \mathcal{J}(q), \frac{ \sin l t }{\sqrt{\pi (1+l^{2})}}\mathbf{a}_{j}\right)_{H^{1}} \notag \\
&= - \mathcal{J}'(q) \left(\frac{ \sin l t }{\sqrt{\pi (1+l^{2})}}\mathbf{a}_{j}\right) \notag \\
&=- \int_{0}^{2 \pi } \dot{q}\cdot\left( \frac{ l \cos l t }{\sqrt{\pi (1+l^{2})}} \mathbf{a}_{j}\right)+\nabla U(q) \cdot \left(\frac{ \sin l t }{\sqrt{\pi (1+l^{2})}}\right)\mathbf{a}_{j}  dt \label{eqn:xi}
\end{align}
where 
\[ U(q_{1}, q_{2})=\frac{1}{|q_{1}-q_{2}|}+\frac{1}{|q_{1}+q_{2}|}+\frac{1}{2|q_{1}|}+\frac{1}{2|q_{2}|} \qquad ((q_{1}, q_{2}) \in \mathbb{R}^{2} \times \mathbb{R}^{2}).\]
The equation with respect to $\eta_{lj}$  is similarly obtained as follows: 
\begin{align}
\frac{d \eta_{lj}}{d s}= -\int_{0}^{2 \pi } -\dot{q}\cdot\left( \frac{ l \sin l t }{\sqrt{\pi (1+l^{2})}} \mathbf{a}_{j}\right)+\nabla U(q) \cdot \left(\frac{ \cos l t }{\sqrt{\pi (1+l^{2})}}\right)\mathbf{a}_{j}  dt. \label{eqn:eta}
\end{align}
 We approximate $T$ with a finite-dimensional subspace $T_{k}$
 \begin{align*}
  T_{k}&= \left\langle  \frac{ \sin l t }{\sqrt{\pi (1+l^{2})}}\mathbf{a}_{1}, \frac{\cos l t }{\sqrt{\pi(1+l^{2})}} \mathbf{a}_{2} , 
  \frac{ \cos l t }{\sqrt{\pi (1+l^{2})}}\mathbf{a}_{3}, \frac{\sin l t }{\sqrt{\pi(1+l^{2})}} \mathbf{a}_{4} 
  \mid 0\le l \le k\right \rangle.
\end{align*}
We restrict the equation \eqref{eqn:xi}, \eqref{eqn:eta} onto $T_k$.

We numerically solve these differential equations \eqref{eqn:xi}-\eqref{eqn:eta}
by using the Euler method with randomly taken initial conditions in $\Gamma$.
All orbits converge to the super-eight orbit as $s$ increases.
The numerical computation shows that the minimizer is the Gerver's super-eight.

The figure stands for one example.
\begin{figure}[htbp]
\begin{center}
\includegraphics[width=12cm]{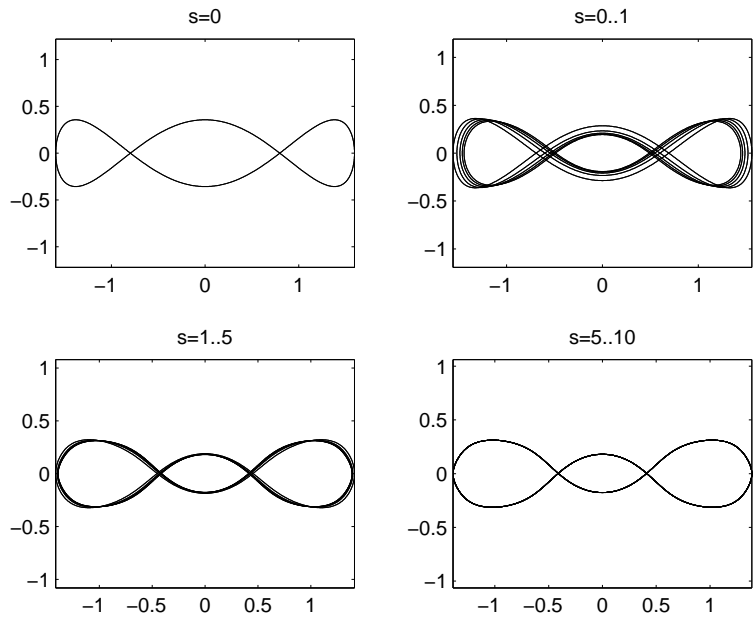}
\caption{Numerical result}
\quote{
This figure stands for one orbit of \eqref{Fin}.
We computed the flow of \eqref{Fin} with many other initial condition 
whose braid is same as one of the super-eight. 
All orbits which the author tried converged to the super-eight orbit. 
}
\end{center}
\end{figure}
\\~\\
{\bf Acknowledgement} 
The author  thanks Professor Kazunaga Tanaka for his helpful comments.
The author is  supported  by the Japan Society for the Promotion of Science (JSPS),
Grant-in-Aid for Young Scientists (B) No. 40467444.

\end{document}